\documentclass[11pt]{article}
\usepackage{setspace}
\usepackage{cprotect}
\usepackage{amsmath,amssymb}
\usepackage{amsthm}
\usepackage[noend]{algorithmic}
\usepackage[ruled,vlined]{algorithm2e}
\usepackage{url}
\usepackage{fullpage}
\usepackage{makeidx}
\usepackage{enumerate}
\usepackage{fullpage}
\usepackage{graphicx,float,psfrag,epsfig,caption,subcaption}
\usepackage{epstopdf}
\usepackage{color}

\usepackage[mathscr]{euscript}

\DeclareSymbolFont{rsfs}{U}{rsfs}{m}{n}
\DeclareSymbolFontAlphabet{\mathscrsfs}{rsfs}

\numberwithin{equation}{section}

\newtheoremstyle{myexample} 
    {\topsep}                    
    {\topsep}                    
    {\rm }                   
    {}                           
    {\bf }                   
    {.}                          
    {.5em}                       
    {}  

\newtheoremstyle{myremark} 
    {\topsep}                    
    {\topsep}                    
    {\rm}                        
    {}                           
    {\bf}                        
    {.}                          
    {.5em}                       
    {}  

\newtheorem{claim}{Claim}[section]
\newtheorem{lemma}[claim]{Lemma}
\newtheorem{fact}[claim]{Fact}

\newtheorem{theorem}{Theorem}

\newtheorem{corollary}[claim]{Corollary}
\newtheorem{definition}[claim]{Definition}

\theoremstyle{myremark}
\newtheorem{remark}{Remark}[section]

\theoremstyle{myremark}

\theoremstyle{myexample}
\newtheorem{example}[remark]{Example}

\def\sT{{\sf T}}
\def\<{\langle}
\def\>{\rangle}

\def\prob{{\mathbb P}}
\def\integers{{\mathbb Z}}
\def\naturals{{\mathbb N}}
\def\E{{\mathbb E}} 

\def\de{{\rm d}}

\def\sTV{\mbox{\tiny\rm TV}}
\def\sUB{\mbox{\tiny\rm UB}}
\def\reals{\mathbb{R}}

\def\normal{{\sf N}}

\def\cA{{\mathcal{A}}}
\def\cB{{\mathcal{B}}}
\def\cG{{\mathcal{G}}}
\def\cH{{\mathcal{H}}}
\def\cI{{\mathcal{I}}}
\def\cJ{{\mathcal{J}}}

\def\cP{{\mathcal{P}}}

\def\cL{{\mathcal{L}}}
\def\cZ{{\mathcal{Z}}}

\def\bg{{\boldsymbol g}}
\def\bB{{\boldsymbol B}}
\def\bM{{\boldsymbol M}}
\def\bE{{\boldsymbol E}}
\def\by{{\boldsymbol y}}

\def\bT{{\boldsymbol T}}
\def\bQ{{\boldsymbol Q}}

\def\bG{{\boldsymbol G}}
\def\bD{{\boldsymbol D}}

\def\bI{{\boldsymbol I}}

\def\bx{{\boldsymbol x}}

\def\bT{{\boldsymbol T}}
\def\bX{{\boldsymbol X}}
\def\bY{{\boldsymbol Y}}
\def\bZ{{\boldsymbol Z}}

\def\btheta{{\boldsymbol \theta}}
\def\btau{{\boldsymbol \tau}}

\def\bSigma{{\boldsymbol \Sigma}}
\def\bU{{\boldsymbol U}}
\def\bV{{\boldsymbol V}}
\def\tbU{\tilde{\boldsymbol U}}
\def\tbV{\tilde{\boldsymbol V}}
\def\tbX{\tilde{\boldsymbol X}}
\def\id{{\boldsymbol I}}

\def\tT{\widetilde{T}}

\def\bA{{\boldsymbol A}}
\def\bB{{\boldsymbol B}}

\def\bg{{\boldsymbol g}}

\def\bxi{{\boldsymbol \xi}}

\def\Group{{\mathfrak G}}
\def\O{{\rm O}}
\def\SO{{\rm SO}}
\def\U{{\rm U}}
\def\U1{{\rm U}(1)}
\def\rQ{{\rm Q}}
\def\Tr{{\rm Tr}}

\def\MSE{{\rm MSE}}
\def\htheta{\hat{\theta}}
\def\eps{{\varepsilon}}

\def\rE{{\rm E}}
\def\rP{{\rm P}}
\def\conn{\sim_{U}}

\title{Group Synchronization on Grids}
\author{Emmanuel Abbe\thanks{Program in Applied and Computational    Mathematics, and EE Department, Princeton University}, \;\;Laurent
  Massoulie\thanks{Inria, MSR-Inria Joint Centre}, \;\;
Andrea 
Montanari\thanks{Department of Electrical Engineering and Department of Sttistics, Stanford University}, \;\; Allan Sly\thanks{Department of Mathematics, Princeton University},
 \;\;  Nikhil Srivastava\thanks{Department of Mathematics, University of California, Berkeley}}

\begin{document}

\maketitle

\begin{abstract}
Group synchronization requires to estimate unknown elements $(\btheta_v)_{v\in V}$ of a compact group $\Group$
associated to the vertices of a graph $G=(V,E)$, using noisy observations of the group differences associated to  the edges.
This model is relevant to a variety of applications ranging from structure from motion in computer vision to graph localization and positioning,
to certain families of community detection problems. 

We focus on the case in which the graph $G$ is the $d$-dimensional grid. Since the unknowns 
$\btheta_v$ are only determined up to a global action of the group, we consider the following weak recovery question.
Can we determine the group difference $\btheta_u^{-1}\btheta_v$ between far apart vertices $u, v$ better than by random guessing?
We prove that weak recovery is possible (provided the noise is small enough) for $d\ge 3$ and, for certain finite groups, for $d\ge 2$.
Viceversa, for some continuous groups, we prove that weak recovery is impossible for $d=2$. Finally, for strong enough noise, weak recovery is always impossible.
\end{abstract}

\section{Introduction}

In the group synchronization problem, we are given a (countable) graph $G=(V,E)$, a group $\Group$
and, for each edge $(u,v)\in E$, a noisy observation $\bY_{u,v}$. The objective is to estimate group elements
$(\btheta_{v})_{v\in V}$ associated to the vertices $v\in V$, under the assumption that the $\bY_{u,v}$ are noisy observations
of the group difference between the adjacent vertices. Roughly speaking (see below for a precise definition):
\begin{align}
\bY_{uv}=\btheta^{-1}_u\btheta_v +\; \mbox{noise}\, .
\end{align}
In order for the above to be unambiguous, we will assume that an orientation $(u,v)$ is fixed arbitrarily for each edge.

It is useful to introduce two concrete examples.
\begin{example}\label{ex:Z2}
The simplest example is $\Group = \integers_2 = \{(+1,-1),\,\cdot\,\}$, the group with elements $(+1,-1)$ and
operation given by ordinary multiplication (equivalently, the group of integers modulo $2$).
For each edge $(u,v)\in E$ we are given $Y_{uv}$ a noisy observation of $\btheta_u\btheta_v = \btheta_u^{-1}\btheta_v$.
For instance we can assume that, for some $p\in [0,1/2)$,
\begin{align}
\bY_{uv}= \begin{cases}
\btheta_{u}\btheta_v &\;\;\; \mbox{ with probability $1-p$,}\\
-\btheta_{u}\btheta_v &\;\;\; \mbox{ with probability $p$.}
\end{cases}\label{eq:NoiseZ2}
\end{align}
with the $(\bY_{uv})_{(u,v)\in E}$ conditionally independent given $(\btheta_v)_{v\in V}$. 
In other words $\bY_{uv}$ is the output of a binary symmetric channel with flip probability $p$ and input $\btheta_u\btheta_v$.

We will refer to this case as $\integers_2$ synchronization.
\end{example}

\begin{example}\label{ex:Orth}
Consider $\Group = \O(m)$: the group of $m\times m$ orthogonal matrices, with the following  noise model.
Let $(\bZ_{uv})_{(u,v)\in E}$ be an i.i.d. collection of matrices with i.i.d. standard normal entries, and define
\begin{align}
\bY_{uv} = \cP_{\O(m)}(\btheta_u^{-1}\btheta_v+\sigma\bZ_{uv})\, .
\end{align}
Here $\cP_{\O(m)}$ is the projector for the Frobenius norm $\|\cdot \|_{F}$ onto the orthogonal group, namely for a matrix $\bM$ with singular value decomposition
$\bM = \bU\bSigma\bV^{\sT}$, we set $\cP_{\O(m)}(\bM) = \bU\bV^{\sT}$.
\end{example}

Group synchronization plays an important role in a variety of applications.

\noindent \emph{Structure from motion} is a central problem in computer vision:
given multiple images of an object taken from different points of view (and in presence of noise or occlusions) we want to reconstruct the
3-dimensional structure of the object \cite{moulon2013global,chatterjee2013efficient,ozyesil2015robust,wilson2016rotations}. A possible intermediate step towards this goal consists in estimating the relative orientation of the object 
with respect to the camera in each image. This can be formulated as a group synchronization problem over $\Group = \SO(3)$,
whereby $\btheta_u$ describes the orientation of image $u$, and pairwise image registration is used to construct the relative rotations $\bY_{uv}$.

\noindent \emph{Graph localization and positioning.} Consider a set of nodes with positions $\bx_1,\dots,\bx_n\in\reals^d$. We want 
to reconstruct the nodes positions from noisy measurements of the pairwise distances $\|\bx_u-\bx_v\|_2$. This question
arises in sensor network positioning \cite{hightower2001location,oh2010sensor},
imaging  \cite{cucuringu2012eigenvector,singer2011three}, manifold learning \cite{tenenbaum2000global}, to name only a few applications.
It is often the case that measurements are only available  for pairs $u,v\in [n]$ that are close enough, e.g. only if
$\|\bx_u-\bx_v\|_2\le \rho$ for $\rho$ a certain communication range \cite{singer2008remark,javanmard2013localization}. 

Graph localization can be interpreted as a group synchronization problem in multiple ways. First, we can interpret the unknown position $\bx_v$
as a translation and hence view it as a synchronization problem over the group of translations in $d$ dimensions. Alternatively we can adopt
a divide-and-conquer approach following \cite{cucuringu2012eigenvector}. First, we consider cliques in the graph and find their relative positions. Then we reconstruct the relative 
orientations of various cliques, which can be formulated as an $\SO(d)$ synchronization problem.

\noindent \emph{Community detection and the symmetric stochastic block model.} The $k$-groups symmetric stochastic block model is a random graph over $n$ 
vertices generated as follows \cite{moore2017computer,abbe2017community}. First, partition the vertex set into $k$ subsets of size $n/k$, uniformly at random. Then connect vertices independently, conditional on the partition.
Two vertices are connected with probability $p$ if they belong to the same subset, and with a smaller probability $q<p$ otherwise.
Given a realization of this graph, we would like to identify the partition. This problem is in fact closely related to synchronizations over 
$\integers_k$ (the group of integers modulo $k$). Extensions of the stochastic block model where edges are endowed with labels have also been considered \cite{lelarge2015}. In particular the so-called censored block model considered in \cite{saade2015} corresponds precisely to Example \ref{ex:Z2} on an Erd\H{o}s-R\'enyi graph.

The literature on group synchronization is fairly recent and rapidly growing. The articles \cite{singer2011angular,wang2013exact} discuss it in a variety of applications 
and propose several synchronization algorithms, mostly based on spectral methods or semidefinite programming (SDP) relaxations.
Theoretical analysis --- mostly in the case of random (or complete) graphs $G$ ---
is developed in
\cite{abbe2014decoding,boumal2014cramer,javanmard2016phase,perry2016message}.
Most of these studies use  perturbation theoretic arguments which crucially
rely on the fact that the Laplacian (or connection Laplacian,
\cite{bandeira2013cheeger}) of the underlying graph has a spectral gap. This
paper shows that nontrivial recovery is possible even in the absence of a
spectral gap, as in the case of grids with $d\ge 3$, whose
Laplacian pseudoinverses have appropriately bounded trace rather than norm.

In the present paper we are interested in $G$ being the $d$-dimensional grid\footnote{The case $d=1$ is somewhat trivial.}, $d\ge 2$. 
Namely, $V=\integers^d$, and --to be definite-- we orient edges in the positive direction:
\begin{align}
E \equiv \big\{ (x,y):\; y-x\in \{e_1,\dots,e_d\}\, \big\}\, ,
\end{align}
where $e_i=(0,\cdots,0,1,0,\dots,0)$ is the $i$-th element of the canonical basis in $\reals^d$. We expect other $d$-dimensional graphs (e.g. 
random geometric graphs) to present a qualitatively similar behavior.

 By construction, we can hope to determine the 
unknowns $(\btheta_x)_{x\in\integers^d}$ only up to a global action by a group element. In other words, we cannot distinguish between 
$(\btheta_x)_{x\in\integers^d}$ and $(\bg\btheta_x)_{x\in\integers^d}$ for some $\bg\in\Group$. We thus ask the following \emph{weak recovery question}:
\begin{quote}
\emph{Is it possible to estimate $\btheta_x^{-1}\btheta_y$ better than random guessing, as $\|x-y\|_2\to\infty$?}
\end{quote}
Note that, in absence of noise (i.e. if $\bY_{uv} = \btheta_u^{-1}\btheta_v$ exactly), the answer is always positive:
we can multiply the observations $\bY_{uv}$'s along any path connecting $x$ to $y$ to reconstruct exactly  $\btheta_x^{-1}\btheta_y$. However
for any arbitrarily small noise level, errors add up along the path and this simple procedure is equivalent to random guessing for $\|x-y\|_2\to\infty$.
The weak recovery question hence amounts to asking whether we can avoid error propagation.

Focusing on the case of compact matrix groups, we will present the following main results:
\begin{description}
\item[Low noise, $d\ge 3$.] For sufficiently low noise, we prove that weak recovery is possible for $d\ge 3$ and any group.
\item[High noise.] Vice-versa, weak recovery is impossible in any dimension at sufficiently high noise (or for $d=1$ at any positive noise).
\item[Discrete groups.] For the special case of $\integers_2$-synchronization, we prove that weak recovery is possible (at low enough noise) for all $d\ge 2$.
We expect the same to hold generally for other discrete groups.
\item[Continuous groups, $d=2$.] Vice-versa, for the simplest continuous group, $\SO(2)$, we prove that weak recovery is impossible for $d=2$.
\end{description}
The above pattern is completely analogous to the one of phase transitions in spin models within statistical physics \cite{georgii2011gibbs}.
We refer to Section \ref{sec:Bayes} for a discussion of the connection with statistical physics.

The rest of the paper is organized as follows. Section \ref{sec:Results} presents formal definitions and statements of our main results.
In order to achieve optimal synchronization, it is natural to consider the Bayes posterior of the unknowns $(\btheta_v)_{v\in V}$, cf.
Section \ref{sec:Bayes}. While this does not lead directly to efficient algorithms, it clarifies the connection with statistical physics.
Some useful intuition can be developed by considering the  case\footnote{Strictly speaking, this is not
a special case of the problem studied in the rest of the paper, because $\Group=\reals$ is not a compact group.} in which $\theta_v\in\reals$ and 
$Y_{uv} =\theta_v-\theta_u+Z_{uv}$ with $(Z_{uv})_{(u,v)\in E}$ i.i.d. noise. This can be treated by elementary methods, cf. Section \ref{sec:Toy}.
Finally, Section \ref{sec:ProofMoments} and \ref{sec:ProofZ2} prove our positive results (reconstruction is possible) with other proofs deferred to
the appendices.

\vspace{0.5cm}

\noindent{\bf Notations.} Throughout the paper we use boldface symbols (e.g. $\btheta_x$, $\bY_{xy}$) to denote elements of the group $\Group$,
and normal symbols for other quantities (including vectors and matrices).

\section{Main results}
\label{sec:Results}

As mentioned above, $G=(V,E)$ will be the infinite $d$-dimensional grid, and  $\Group$ a compact matrix group.
Without loss of generality, we will assume $\Group\subseteq \O(m)$ (the group of $m\times m$ orthogonal matrices).
We attach to each vertex $x\in V$ an element $\btheta_x\in\Group$ which may be
deterministic or random chosen independently from some distribution. 

We are given observations $\bY = (\bY_{xy})_{(x,y)\in E}$, $\bY_{xy}\in \Group$, that are conditionally independent
given $\btheta$. We assume that observations are unbiased in the following sense:
\begin{align}
\E\{\bY_{xy}|\btheta\} = \lambda\, \btheta_x^{-1}\btheta_y\, ,\label{eq:Unbiased}
\end{align}
where the parameter $\lambda\in [0,1]$ is a natural measure of the signal-to-noise ratio. In particular, $\lambda=1$ corresponds
to noiseless observations. The two examples given in the introduction fit this general definition:
\begin{itemize}
\item For $\integers_2$ synchronization (cf. Example \ref{ex:Z2}) we have $\E\{\bY_{xy}|\btheta\} = (1-2p)\, \btheta_x^{-1}\btheta_y$, and
therefore $\lambda = (1-2p)$.
\item For $\O(m)$ synchronization (cf.  Example \ref{ex:Orth}) we have $\E\{\bY_{xy}|\btheta\} = \lambda(\sigma^2)\, \btheta_x^{-1}\btheta_y$
where $\sigma^2\mapsto \lambda(\sigma^2)$ is a continuous function on $[0,\infty)$ with $\lambda(\sigma^2)\to 1$ as $\sigma^2\to 0$ and
$\lambda(\sigma^2)\to 1$ as $\sigma^2\to \infty$ (see Appendix \ref{app:Lambda}).
\end{itemize}
A simple mechanism to produce the noisy observations $\bY_{xy}$ consists in introducing a probability kernel $\rQ$ on $\Group$ 
and stipulate that, for each edge $(x,y)$, 
\begin{align}
\prob (\bY_{x,y}\in\,\cdot\,|\btheta)  =\prob (\bY_{x,y}\in\,\cdot\,|\btheta_x^{-1}\btheta_y) =\rQ(\,\cdot\,|\btheta_x^{-1}\btheta_y )\, .
\end{align}
In other words, all observations are obtained by passing $\btheta_x^{-1}\btheta_y$ through the same noisy channel.
While our results do not necessarily assume this structure, both of the examples given above are of this type.

An estimator is a collection of measurable functions $T_{uv}:\bY\mapsto T_{u,v}(\bY)\in\Group$ indexed by all vertex pairs $u,v\in V$
(here $\bY=(\bY_{xy})_{(x,y)\in E}$ denotes the set of all observations).
\begin{definition}
We say that the weak recovery problem is solvable for the probability distribution $\prob$ over $(\btheta,\bY)$ defined above if
there exists an estimator $T$, and $\eps>0$, such that
\begin{align}
\liminf_{\|x-y\|\to\infty}\Big\|\prob\big(\btheta_xT_{xy}(\bY)\btheta_y^{-1}\in \;\cdot\;\big)- \prob_{{\rm Haar}}\big( \; \cdot\;\big)\Big\|_{\sTV}\ge \eps>0\, .
\label{eq:WeakRecovery}
\end{align}
\end{definition}

Our first result establishes  that the problem is solvable if noise is small enough in $d\ge 3$ dimensions.
\begin{theorem}\label{thm_moments}
If $d\ge 3$, then there exists $\lambda_{\sUB}\in (0,1)$ such that, if $\lambda> \lambda_{\sUB}$ then the weak recovery problem is solvable.
\end{theorem}

If noise is strong enough, the problem becomes unsolvable.
\begin{theorem}\label{thm:Noisy}
Assume that: 
\begin{enumerate}
\item $\prob (\bY_{x,y}\in\,\cdot\,|\btheta)=\prob (\bY_{x,y}\in\,\cdot\,|\btheta_x^{-1}\btheta_y)$.
\item $\prob (\bY_{x,y}\in\,\cdot\,|\btheta_x^{-1}\btheta_y)$   has density $q(\by|\btheta_0)$, $\btheta_0\in\Group$ 
with respect to the Haar probability measure.
\end{enumerate}
Let $p_c(d)\in (0,1]$ the critical threshold for percolation on the $d$-dimensional grid. If
\begin{align}
\inf_{\by,\btheta_0}  q(\by|\btheta_0)> 1-p_c(d)\, ,
\end{align}
then the weak recovery problem is not solvable.

In particular, for $d=1$, the recovery is not solvable as soon as the noise is strictly positive $\inf_{\by,\btheta}  q(\by|\btheta)> 0$
(since $p_c(d=1)=1$).
\end{theorem}

In $d=2$ the situation is more complicated. For certain discrete groups the problem is solvable at low enough noise: we consider here
the case $\Group = \integers_2$, but we expect the same conclusion to
hold more generally. A result related to the next one 
was established in \cite{horiguchi1982existence} using a Peierls argument (Section \ref{sec:Bayes} outlines the connection with the statistical physics formulation). 
We present here an independent proof that also provides an efficient 
recovery algorithm.
\begin{theorem}\label{t:Z2NoiseZ2}
Consider $d = 2$, and $\Group = \integers_2$, with uniform flip probability $p$. Then there exists $p_*\in (0,1)$ such that, if $p\le p_*$ then
the weak recovery problem is solvable.
\end{theorem}

On the contrary, we expect that weak recovery is not possible in $d=2$ dimensions, for continuous groups even for very weak noise.
This is analogous to the celebrated Mermin-Wagner theorem in statistical mechanics  \cite{mermin1966absence,mermin1967absence}.
For the sake of simplicity, we focus on the case of $\Group = \SO(2)$ which is isomorphic to $\U1$, the group of complex variables
of unit modulus, with ordinary multiplication.  Let $Z$ a $\U1$-valued random variable with density $g$ satisfying
\begin{equation}\label{e:U1NoiseA}
g\in C^2, \quad \inf_{s\in [0,2\pi]} g(e^{is}) >0 \, .
\end{equation}
We consider observation on the edges corrupted by multiplicative noise
\begin{equation}\label{e:U1NoiseB}
\bY_{xy} =  \btheta_x^{-1}\btheta_y Z_{xy}\, ,
\end{equation}
where $(Z_{xy})_{(x,y)\in E}\sim_{iid}g$.
\begin{theorem}\label{t:Z2ctsNonRecon}
If $d=2$ and $\Group = \SO(2)$ with noise model satisfying \eqref{e:U1NoiseA} and \eqref{e:U1NoiseB}, then the weak recovery problem is not solvable.
\end{theorem}

\section{Bayesian posterior and connection to statistical physics}
\label{sec:Bayes}

In this section it is convenient to assume a more general model in which the observations $\bY_{xy}\in \reals^{m\times m}$ are not necessarily 
elements of the matrix group $\Group$. We assume that the conditional distribution of the observations $\bY_{xy}$
given the unknowns $\btheta_x$ is absolutely continuous with respect to a  reference measure $\prob_{\#}$ (independent of $\btheta$).
In practice, we will take $\prob_{\#}$ to be either the Haar measure on $\Group$, or the Lebesgue measure on $\reals^{m\times m}$. 
We denote the corresponding density by
\begin{align}
\frac{\de \prob}{\de\prob_{\#}} (\bY_{x,y}|\btheta) = \frac{1}{{\cZ_0}}\, \exp\big\{-u(\btheta_x^{-1}\btheta_y;\bY_{xy})\big\}\, .
\end{align}
where $u:\Group\times \reals^{m\times m}\to \reals\cup \{+\infty\}$ is a measurable function bounded below.
Applying Bayes formula, we can write the posterior $\mu_{\bY} (B) = \prob(\btheta\in B|\bY)$ as
\begin{align}
\mu_{\bY}(\de\btheta) = \frac{1}{\cZ(\bY)}\, \exp\Big\{-\sum_{(x,y)\in E} u(\btheta_x^{-1}\btheta_y;\bY_{xy})\Big\} \; \mu_0(\de\btheta)\, ,\label{eq:Gibbs}
\end{align}
where $\mu_0(\de\btheta) = \mu_0(\de\btheta_1)\cdots\mu_0(\de\btheta_n)$ is the product Haar measure over the unknowns and
$\cZ(\bY)$ is a normalization constant. The joint distribution (\ref{eq:Gibbs}) takes the form of a Gibbs measure on the graph $G$.
\begin{remark}
For Eq.~(\ref{eq:Gibbs}) to make sense, the graph $G$ needs to be finite. However, the Bayesian interpretation implies immediately
that quantities of interest have a well defined limit over increasing sequences of graphs. In particular, we can take $G$ to be 
the finite grid with vertex set $V = \{-L,\dots,L\}^d$, and edges $E=\{(x,y)\in V\times V: \; y-x\in \{e_1,\dots,e_d\}\}$. 
Then the quantity 
\begin{align} 
\sup_{T_{xy}(\,\cdot\,)}\Big\|\prob\big(\btheta_xT_{xy}(\bY)\btheta_y^{-1}\in \;\cdot\;\big)- \prob_{{\rm Haar}}\big( \; \cdot\;\big)\Big\|_{\sTV} 
\end{align}
is obviously non-decreasing in $L$ (because larger $L$ corresponds to a larger class of estimators) and hence admits a well defined limit.
We will refer succinctly to this $L\to\infty$ limit as the model on `the $d$-dimensional grid'.
\end{remark}

In the rest of this section, it will be useful to distinguish between the arguments of the posterior density (that we will keep denoting by $(\btheta_x)_{x\in V}$),
and the true unknowns that we will denote by $(\btheta_{0,x})_{x\in V}$. We further assume that the function $u$ satisfies
\begin{align}
u(\btheta \btau;Y)=u(\btau;\btheta^{-1}Y)=u(\btheta;Y\btau^{-1}).
\end{align}
for any $\btheta$, $\btau\in \Group$ and any $Y\in\reals^{m\times m}$. This condition is verified by all of our examples.
Thanks to this symmetry, for any $\{\btau_x\}_{x\in V}$ and any $\bY$, the distribution $\mu_{\bY}(\,\cdot\,)$ of $\btheta$ in \eqref{eq:Gibbs} coincides with that of $\{\btheta_x\btau^{-1}_x\}_{x\in V}$ where $\btheta$ is distributed according to $\mu_{\tilde{\bY}}(\cdot)$, and $\tilde{\bY}_{xy}=\btau_x \bY_{xy}\btau^{-1}_y$. By taking $\btau_x=
\btheta_{0,x}$  for all $x$, we can assume that $\btheta_{0,x}=\id_m$ 
for all $x$, which then leads to the $(\bY_{xy})_{(xy)\in E}$ being i.i.d. with common distribution
\begin{align}
\bY_{xy}\sim  \frac{1}{{\cZ_1}}\, \exp\big\{-u(\id_m;\bY_{xy})\big\} \prob_{\#}(\de\bY_{xy})\, . \label{eq:DistrNishimori}
\end{align}
In the jargon of statistical physics,  Gibbs measures of the form (\ref{eq:Gibbs}) with associated parameters distribution (\ref{eq:DistrNishimori}) are
known as spin-glasses on the `Nishimori line.' These were first introduced for the case $\btheta_x\in\{+1,-1\}$ \cite{nishimori1981internal} and 
subsequently generalized to other groups in  \cite{georges1985exact}. 
Several results about spin glasses on the Nishimori line were derived in \cite{ozeki1993phase,nishimori2001statistical} and the connection with Bayesian statistics was emphasized in 
\cite{iba1999nishimori,andrea2008estimating}.
The weak recovery phase transition corresponds to a paramagnetic-ferromagnetic phase transition in physics language.
\begin{example}\label{ex:RBIM}
The simplest example is the so-called \emph{random bond Ising model} which is obtained by taking $\btheta_x\in \{+1,-1\}$
and
\begin{align}
\mu_{\bY}(\btheta) = \frac{1}{\cZ(\bY)}\, \exp\Big\{\beta\sum_{(x,y)\in E} \bY_{xy}\btheta_x\btheta_y\Big\} \, ,\label{eq:Ising}
\end{align}
where $\bY_{xy} = +1$ with probability $1-p$ and $\bY_{xy} = -1$ with probability $p$. The Nishimori line is given by the condition $\beta = (1/2)\log((1-p)/p)$.
It is easy to see that this is equivalent to the Bayes posterior for the $\integers_2$ synchronization model of Example \ref{ex:Z2}, if we take $\btheta_{0,x}=+1$.

This model has attracted considerable interest within statistical physics. In particular, high-precision numerical estimates of the phase transition 
location yield $p_c\approx 0.1092$ (in $d=2$) and $p_c \approx 0.233$ (in $d=3$) \cite{parisen2009strong,hasenbusch2007critical}.
\end{example}

\begin{example}\label{ex:Om}
Take $\Group = \O(m)$ (the group of orthogonal matrices), and assume
\begin{align}
\bY_{xy} = \btheta_{0,x}^{-1}\btheta_{0,y} +\sigma\, \bZ_{xy}
\end{align}
where $\bZ_{xy}$ is a noise matrix with i.i.d. entries $(\bZ_{xy})_{ij}\sim \normal(0,1)$. 
This model is analogous to the one of Example \ref{ex:Orth}, although we do not project observations onto the orthogonal group.
 
After a simple calculation, the Gibbs measure  (\ref{eq:Gibbs}) takes the form
\begin{align}
\mu_{\bY}(\de\btheta) = \frac{1}{\cZ(\bY)}\, \exp\Big\{\beta\sum_{(x,y)\in E} \Tr\big(\btheta_x\bY_{xy}\btheta_y^{\sT}\big)\Big\} \; \mu_0(\de\btheta)\, ,\label{eq:GibbsOm}
\end{align}
where $\beta=1/\sigma^2$. By the symmetry under $\O(m)$ rotations, for the purpose of analysis we can assume $\bY_{xy} = \id_m+\sigma\, \bZ_{xy}$
which is the usual setting in physics. 
\end{example}
\begin{example}\label{ex:XY}
In the case $\Group = \SO(2)$ we can identify $\btheta_x$ with an angle in $[0,2\pi)$, and  let
\begin{align}
\bY_{xy} = \btheta_{0,y} -\btheta_{0,x}+ \bZ_{xy},\; \;\;\;\; \mod 2\pi\, ,
\end{align}
where $\bZ_{xy}$ is noise with density proportional to $\exp(-u(z))$ for $u(z)$ a periodic function bounded below.

The Gibbs measure (\ref{eq:Gibbs}) takes the form
\begin{align}
\mu_{\bY}(\de\btheta) = \frac{1}{\cZ(\bY)}\, \exp\Big\{-\sum_{(x,y)\in E} u\big(\bY_{xy}-\btheta_y+\btheta_x\big)\Big\} \; \mu_0(\de\btheta)\, . \label{eq:GibbsXY}
\end{align}
For the purpose of analysis we can assume $\bY_{xy} = \bZ_{xy}$. This is known as the `XY model' in physics.
\end{example}

Our results have direct implications on these models that we summarize in the following statement.
\begin{corollary}
Consider the Gibbs measure (\ref{eq:Gibbs}) on the $d$-dimensional grid, with parameters $\bY_{xy}\in \Group$ distributed  according to Eq.~(\ref{eq:DistrNishimori})
and satisfying Eq.~(\ref{eq:Unbiased}). Then, the following hold:
\begin{enumerate}
\item For $d\ge 3$, and $\Group\subseteq \O(m)$ is any compact matrix group, then
there exists $\lambda_{{\rm UB}}<1$ such that the model is in a ferromagnetic phase for any $\lambda> \lambda_{{\rm UB}}$.
\item For the case of Example \ref{ex:RBIM}  (i.e. $\Group =\integers_2$) and $d\ge 2$,
there exists $p_*\in (0,1)$ such that the model is in a ferromagnetic phase for any $p\le p_*$.
\item For the case of Example \ref{ex:XY}  (i.e. $\Group = \SO(2)$) and $d=2$ the 
model is not in a ferromagnetic phase  provided $z\mapsto u(z)$ is bounded.
\item For any group $\Group$, $d\ge 2$,   there exists a constant $c(d)$ such that, if $\|u\|_{\infty}\le c(d)$
then the model is not in a ferromagnetic phase.
\end{enumerate}
Furthermore point 1 applies to Example~\ref{ex:Om} as well.
\end{corollary}
\begin{proof}
These statements are merely a translation of Theorems \ref{thm_moments}, \ref{thm:Noisy}, \ref{t:Z2NoiseZ2}, \ref{t:Z2ctsNonRecon}
for the case in which channel observations take values in $\Group$. For the case in Example~\ref{ex:Om}, note that we can
always project $\bY_{xy}$ onto the group $\O(m)$, hence recovering the setting of Example \ref{ex:Orth}. Since weak recovery is possible in the latter,
it is also possible in the former.
\end{proof}
As already pointed out in Section \ref{sec:Results}, the existence of a ferromagnetic phase for Example \ref{ex:RBIM}  (i.e. $\Group =\integers_2$) was
already obtain in \cite{horiguchi1982existence}. Note however that \cite{horiguchi1982existence} estabilish existence of a spontaneous magnetization, 
while here we prove the existence of long range point-to-point correlation, which is equivalent to weak recovery.

\section{A toy example}
\label{sec:Toy}

It is instructive to consider a simple example in which $\Group =\reals$ is the group of translations
on the real line. This case does not fit the framework of the rest of this paper, but presents the same dichotomy between 
$d=2$ and $d\ge 3$ and can be solved by elementary methods. 

Throughout this section, we adopt additive notation, and hence the observation on edge $(x,y)$ takes the form
\begin{align}
Y_{x,y} = \theta_y-\theta_x+Z_{x,y}\, .\label{eq:Translation}
\end{align}
 where $\{Z_{xy}\}_{(x,y)\in E}$ are i.i.d. random variables with mean $0$ and variance $\sigma^2$.

To simplify our treatment, we assume the graph to be the discrete torus, with vertex set $V = \{1,2,\dots,L\}^d$ and
edges $E= \{(x,x+e_j): \;\; x\in V, \, j\in\{1,\dots,d\}\}$ (where we identify $L+1$ with $1$). 
 Denoting by $D$ the difference operator on 
$G$, the observation can be written as
\begin{align}
Y = D\theta+Z\, .
\end{align}
As usual, $\theta$ can be determined only up to a global shift. To resolve this ambiguity, it is convenient to assume that $\theta$ is centered:
 $\<\theta,1\>=0$.
Consider the least square estimator $\htheta(Y) = D^{\dagger}Y=DL^\dagger Y$ where $\dagger$
denotes the pseudoinverse. A standard calculation \cite[Theorem 13.13]{wasserman2013all} yields the following
formula for the mean square error
\begin{align}
\MSE(L,\sigma^2) &\equiv \frac{1}{L^d}\E\{\|\btheta(Y)-\theta\|_2^2\} \\
& = \frac{\sigma^2}{L^d}\Tr_0\big((D^{\sT}D)^{\dagger}\big) =
\frac{\sigma^2}{L^d}\Tr_0\big(\cL^{\dagger}\big) \, .
\end{align}
Here we denoted by $\cL = D^{\sT}D$ the Laplacian of $G$ and by $\Tr_0$ the trace on the subspace orthogonal to the all-ones vector.
The eigenpairs of the Laplacian are \cite{mohar1997some}:
\begin{align}
v(p)_x = \frac{1}{L^{d/2}}\, e^{i\<p,x\>}\, ,\;\;\;\; \lambda(p) = \sum_{i=1}^d[2-2\cos(p)]\, ,\\
p\in B_L\equiv \Big\{\frac{2\pi}{L}(n_1,\dots,n_d):\; n_i\in \{0,\dots, L-1\}\Big\}\, .
\end{align}
Hence
\begin{align}
\MSE(L,\sigma^2) = \frac{\sigma^2}{L^d} \sum_{p\in B_L\setminus\{0\}} \frac{1}{\lambda(p)}\, .
\end{align}
For large $L$, the sum can be estimated by approximating it via Riemann integrals to yield the following fact.
\begin{fact}
The mean square error of least-square estimation within the translation synchronization model of Eq.~(\ref{eq:Translation}) is
\begin{align}
\frac{1}{\sigma^2}\, \MSE(L,\sigma^2) =
\begin{cases}
\frac{L}{12}+O_L(1) & \mbox{ for $d=1$,}\\
\frac{1}{2\pi}\log L+O_L(1) & \mbox{ for $d=2$,}\\
C(d) +o_L(1)& \mbox{ for $d\ge 3$.}
\end{cases}
\end{align}
where $C(d)$ is a dimension dependent constant. 
\end{fact}

We observe that this qualitative behavior is the same that we obtain for continuous compact groups, cf. Theorem
\ref{thm_moments} and Theorem \ref{t:Z2ctsNonRecon}: the weak recovery problem is solvable only for $d\ge 3$.

\section{Proof of Theorem \ref{thm_moments}}
\label{sec:ProofMoments}

Throughout this section we assume  a probability distribution $\prob$ over $\btheta$, $\bY$ satisfying the unbiasedness condition Eq.~(\ref{eq:Unbiased}).
For most of our analysis, we consider general estimators $T_{uv}:\bY\mapsto T_{uv}(\bY)\in\reals^{m\times m}$ whose output is not necessarily in $\Group$,
and let $\bT_{uv} = T_{uv}(\bY)$ (as projecting them into $\Group$ 
at the end can only increase their accuracy). Also, we set $u(n) = (n,\dots,n)\in\integers^d$, and denote by $\cP_+$ the set of infinite increasing paths in
the grid, that start at $0$.

Throughout the proof, we will use repeatedly the following two elementary facts. First, for any two matrices $\bA,\bB$,$\|\bA\bB\|_F \le \|\bA\|_F\|\bB\|_F$.
Second, if $\bB$ is an orthogonal matrix, then $\|\bA\bB\|_F = \|\bA\|_F$.

We start by defining the estimator $T_{x,y}(\bY)$  for $x=0$, $y=u(n)$, and will then generalize it to other pairs $x,y$.
\begin{lemma}\label{lemma:MomentsFirst}
Consider $d\ge 3$. Then there exists an estimator $T=(T_{u,v})$, and absolute constants $\lambda_0<1$, $C_0$  such that, for all $\lambda>\lambda_0$ and all even $n$,
\begin{align}
\E \big\{\btheta_0\bT_{0,u(n)}\btheta_{u(n)}^{-1}\big\} & = \id_m\, ,\\
\E\Big\{\big\|\btheta_0\bT_{0,u(n)}\btheta_{u(n)}^{-1}-\id_m\big\|_F^2\Big\}&\le C_0\, m (1-\lambda)\, .
\end{align}
\end{lemma}
\begin{proof}
Benjamini, Pemantle and Peres \cite{benjamini1998unpredictable} construct a probability measure $\mu$ over paths
in $\cP_+$ satisfying the so called exponential intersection property (EIT). Namely, there exist absolute constants $\beta_*<1$, $C_*$
such that 
\begin{align}
(\mu \times \mu) \{( \gamma_1, \gamma_2) \in  \cP_+\times \cP_+: | \gamma_1 \cap \gamma_2| \ge k \} \le C_* \beta_*^k. \label{EITv}
\end{align}
Let $\cP_{+}(v)$ be the set of increasing paths starting at $0$ and ending at $v\in\integers^d$.
For $n$ even, we construct a probability measure $\mu_n$ over $\cP_+(u(n))$ as follows. 
Define the hyperplane  $H(n) = \{(x_1,x_2,x_3)\in \reals^3:\; x_1+x_2+x_3=3n/2\}$, and let $R_n(x_1,x_2,x_3) = (n-x_1,n-x_2,n-x_3)$
denote the reflection with respect to this hyperplane. For $\gamma\sim \mu$, let $\gamma^{(1,n)}$ denote the path obtained by stopping $\gamma$
when it hits $H(n)$, and denote by $\gamma^{(n)}$ its extension obtained by reflecting the  with respect to  $H(n)$. We let
$\mu_n$ be the probability distribution of $\gamma^{(n)}$ (note that $\gamma^{(n)}$ ends at $u(n)$ by construction). It follows immediately from 
Eq.~(\ref{EITv}) that $\mu_n$ satisfies the EIT for some new absolute constants $C,\beta$, that are independent of $n$:
\begin{align}
\mu_n \times \mu_n \{( \gamma_1, \gamma_2) \in  \cP_+(u(n))\times \cP_+(u(n)): | \gamma_1 \cap \gamma_2| \ge k \} \le C \beta^k. \label{EIT_n}
\end{align}

For a path $\gamma \in \cP_+(u(n))$, denote the ordered sequence of directed edges in $\gamma$ by $I_1(\gamma),\dots,I_{3n}(\gamma)$, where $I_{j}(\gamma) \in E$, 
$j \in [3n]$, and define
\begin{align}
\bY_\gamma  &:=  \bY_{I_1(\gamma)} \bY_{I_2(\gamma)} \cdots \bY_{I_{3n}(\gamma)} \, ,\\
\bT_{0,u} &:= \frac{1}{\lambda^{3n}} \rE_{\gamma}(\bY_\gamma)\, ,
\end{align}
where $\rE_\gamma$ denotes expectation with respect to $\mu_n$. 
Note that by the  assumption (\ref{eq:Unbiased}) we have $\E\bY_{\gamma} = \lambda^{3n}\btheta_0^{-1}\btheta_{u(n)}$ for any $\gamma\in \cP_+(u(n))$
and therefore
\begin{align}
\E\bT_{0,u(n)} = \btheta_0^{-1}\btheta_{u(n)}\, .\label{eq:exp}
\end{align}
Observe that if two paths $\gamma_1,\gamma_2$ in $\cP_+(u(n))$ intersect in an
edge $e$ then they must intersect in the same position since the paths are
increasing, i.e.  we must have $e=I_k(\gamma_1)=I_k(\gamma_2)$ for some $k$.
Writing for simplicity $u=u(n)$, and denoting by $\rE_{\gamma_1,\gamma_2}$ expectation with respect to $\gamma_1,\gamma_2\sim_{iid}\mu_n$
\begin{align*}
\E \,  \big\{\bT_{0,u} \bT_{0,u}^{\sT}\big\} &=  \frac{1}{\lambda^{6n}} \, \rE_{\gamma_1,\gamma_2}\E \bY_{\gamma_1} (\bY_{\gamma_2})^\sT\\
&=  \frac{1}{\lambda^{6n}} \,
\rE_{\gamma_1,\gamma_2}\E \bY_{I_1(\gamma_1)}\ldots 
	(\E\bY_{I_{3n}(\gamma_1)} \bY_{I_{3n}(\gamma_2)}^\sT)
	\bY_{I_{3n-1}(\gamma_2)}^\sT\ldots \bY_{I_1(\gamma_2)}^\sT\\
&\stackrel{(a)}{=}  \frac{1}{\lambda^{6n} }  \rE_{\gamma_1,\gamma_2}\lambda^{|\gamma_1|+|\gamma_2| - 2|\gamma_1 \cap \gamma_2|} \bI_m\\
&=   \rE_{\gamma_1,\gamma_2}\lambda^{ - 2|\gamma_1 \cap \gamma_2|} \, \bI_m\,.\label{var}
\end{align*}
where $(a)$ follows by repeatedly applying the identity $\bY_e\bY_e^\sT=I_m$ for any edge $e$, each time an intersection appears, and taking expectation with respect to 
$\bY_{e_1}$, $\bY_{e_2}$ for not repeated edges. 
By this last expression, the trace $\tau$ of $\E \,  \big\{\bT_{0,u} \bT_{0,u}^{\sT}\big\} $ reads $\tau=m\;\E\left(\lambda^{-2X}\right)$ where $X$ is a random variable counting the number of intersections in two paths $\gamma_1$, $\gamma_2$ independently drawn from $\mu_n$. Thus for $\lambda^2 > \beta$, 
\begin{align}
m^{-1}\tau & =  \sum_{x\ge 0} \lambda^{-2x}\left[\prob(X\ge x)-\prob(X\ge x + 1)\right]\\
&=1+\sum_{x>0}\prob(X\ge x)\left[\lambda^{-2x}-\lambda^{-2x+2}\right]\\
&\le 1+(1-\lambda^2)\sum_{x>0}C \left(\beta/\lambda^2\right)^{x}\\
&= 1+(1-\lambda^2)\frac{C\beta}{\lambda^2 - \beta},
\end{align}
where the inequality follows from Eq.~(\ref{EIT_n}). Thus 
\begin{align}
\E\Big\{\big\|\btheta_0\bT_{0,u}\btheta_{u}^{-1}-\id_m\big\|_F^2\Big\}&=\hbox{Tr }\E\left\{\btheta_0\bT_{0,u}\bT_{0,u}^{\sT}\btheta_0^{\sT}\right\}-
2\hbox{Tr }\E\left\{\btheta_0\bT_{0,u}\btheta_{u}^{-1}\right\}+m\\
&= \tau-m\\
&\le (1-\lambda^2)m\frac{C\beta}{\lambda^2 - \beta},
\end{align}
where we used Eq.~\eqref{eq:exp} together with our previous bound on $\tau$. The second statement of the Lemma follows. 
%
%
\end{proof}

\begin{lemma} \label{lemma:MomentsFourth}
Consider any  $d\ge 3$ and fix $\eps>0$.  For $n\in\naturals$, $j\in \{1,\dots, d\}$, let $v(j,n) \equiv n\, e_j$
Then there exists an estimator $T=(T_{u,v})_{u,v\in V}$, and  a constant $\lambda(\eps)<1$,  such that, for all $\lambda>\lambda(\epsilon)$ and all $n$,
\begin{align}
\prob\Big\{\big\|\btheta_0\bT_{0,v(j,n)}\btheta_{v(j,n)}^{-1}-\id_m\big\|_F\ge \eps\Big\}\le \eps\, .
\end{align}
\end{lemma}
\begin{proof}
Without loss of generality, assume $j=1$, and set for simplicity $v(n) = v(j,n)$. Consider first the case of $n$ even and let $w(n) \equiv (n/2,n/2,n/2,0,\dots,0)$.
Let $(\bT^{(*)}_{x,y})_{x,y\in V}$ be the estimator of Lemma \ref{lemma:MomentsFirst} (where we use only the observations on the subgraph induced by the hyperplane
$\{x \in \integers^d:\; x_4=\dots=x_d=0\}$). Define 
\begin{align}
\bT_{0,v(n)} = \bT^{(*)}_{0,w(n)} \bT^{(*)}_{w(n),v(n)}\, . 
\end{align}
From the inequality 
\begin{equation}\label{eq:prod_frob}
1+\|\bX_1\bX_2-\id\|_F\le (1+\|\bX_1-\id\|_F)(1+\|\bX_2-\id\|_F),
\end{equation}
 we get
$$
\|\btheta_0 \bT_{0,v(n)} \btheta_{v(n)}^{-1}-\id\|_F\le \left(1+\|\btheta_0 \bT^{(*)}_{0,w(n)}\btheta^{-1}_{w(n)}-\id\|_F\right)  \left(1+ \|\btheta_{w(n)}\bT^{(*)}_{w(n),v(n)}   \btheta_{v(n)}^{-1}-\id\|_F \right)-1.
$$
By Lemma \ref{lemma:MomentsFirst} and Markov's inequality, the probability that one of the Frobenius norms in the right-hand side exceeds $\delta>0 $ is at most $C_0m(1-\lambda)/\delta^2$. 
Thus with probability at least $1-2C_0m(1-\lambda)/\delta^2$, one has 
$$
\|\btheta_0 \bT_{0,v(n)} \btheta_{v(n)}^{-1}-\id\|_F\le (1+\delta)^2-1.
$$
The right-hand side is at most $3\delta$ for $\delta\le 1$. The announced result follows for the choice
$\lambda(\eps)=1-\eps^3/(18 m C_0)$.
\end{proof}

%
%

We can now prove our main result, that is a strengthening of Theorem \ref{thm_moments}.
\begin{theorem}
Consider any  $d\ge 3$ and fix $\eps>0$.  
Then there exists an estimator $T=(T_{u,v})_{u,v\in V}$, and  a constant $\lambda_d(\eps)<1$,  such that, for all $\lambda>\lambda_d$ and all $n$,
\begin{align}
\prob\Big\{\big\|\btheta_x\bT_{x,y}\btheta_{y}^{-1}-\id_m\big\|_F\ge \eps\Big\}\le \eps\, .
\end{align}
\end{theorem}
\begin{proof}
Without loss of generality, assume $x = 0$. Further, for $j\in \{0,\dots,d\}$ define $w(j)\equiv (y_1,\dots, y_j,0,\dots,0)$. 
In particular, $w(0) = 0$ and $w(d) = y$.
Let $(\bT^{(\#)}_{x,y})$ be the estimator of Lemma \ref{lemma:MomentsFourth}, and  define
\begin{align}
\bT_{0,y} = \bT^{(\#)}_{w(0),w(1)}\bT^{(\#)}_{w(1),w(2)}\cdots \bT^{(\#)}_{w(d-1),w(d)}\, .
\end{align}
By Lemma \ref{lemma:MomentsFourth}, for all $\lambda>\lambda(\eps_0)$ we have
\begin{align}
\prob\Big(\max_{1\le j\le d}\big\|\btheta_{w(j-1)}\bT^{(\#)}_{w(j-1),w(j)} \btheta_{w(j)}^{-1} - \id_m\big\|_F\ge \eps_0\Big)\le d\,\eps_0\, .
\end{align}
By repeated application of Inequality~\eqref{eq:prod_frob}, on the complement of the event in the right-hand side, one has
$$
\| \btheta_0\bT_{0,y}\btheta_y^{-1}  -\id_m\|  \le \prod_{j=1}^d \left(\|\btheta_{w(j-1)}\bT^{(\#)}_{w(j-1),w(j)} \btheta_{w(j)}^{-1}-\id_m   \|   +1\right)-1\le \eps_0 2^d.
$$

The claim follows by taking $\eps_0=\eps/(2^d)$ and by setting $\lambda_d(\eps)=\lambda(\eps/2^d)$.
\end{proof}

\section{Proof of Theorem \ref{t:Z2NoiseZ2}}
\label{sec:ProofZ2}

We give a multi-scale scheme to reconstruct  the unknowns $\btheta =
(\btheta_x)_{x\in \integers_2}$.  Without loss of generality we will consider pairs of vertices $u,v$ in the positive quadrant.
For $k\geq 0$ let $\ell_k = 2^{10k(k+1)}$.  We partition the lattice $\integers^2$ into blocks of side-length $\ell_k$ as follows,
\begin{align}
B_{u}^{(k)} = \{(x_1,x_2)\in \integers^2: u_i=\lceil x_i/\ell_k \rceil\}
\end{align}
Let $\cB^{(k)}$ be the set of blocks at level $k$ and let $D_{u,k}$ denote the unique block in $\cB^{(k)}$ containing $u$.  For each block $B \in \cB^{(k)}$ we will define synchronization random variables $W_B^{(k)}\in \{-1,1\}$ that are measurable with respect to $\{\bY_{xy}\}_{x,y\in B}$.  Our estimate for $\btheta_u \btheta_v^{-1}$ is $\prod_{k\geq 0} W_{D_{u,k}}^{(k)} W_{D_{v,k}}^{(k)}$.  For some large enough $k_\star$ we have that $D_{u,k_\star} = D_{v,k_\star}$ and so $W_{D_{u,k}}^{(k)} W_{D_{v,k}}^{(k)}=1$ for all $k\geq k_\star$.  The product of synchronization variables at $u$ up to level $k$ will be denoted as
\begin{align} 
\tilde{W}_{u}^{(k)}= \prod_{\ell=1}^{k} W_{D_{u,\ell}}^{(\ell)}\, .
\end{align}

We say that two blocks $B, B' \in \cB^{(k)}$ are adjacent (denoted $B \sim B'$) if there exist $x\in B, x\in B'$ such that $(x,x')\in E$.  
In this case there are exactly $\ell_k$ such pairs.  We say that $B\sim B'$ is an honest edge  if the following event holds
\begin{align}
\cA^{(k)}(B,B') = \bigg\{\sum_{x\in B,x'\in B'} \bY_{xx'}\btheta_x \btheta_{x'} \geq \frac{9}{10}\ell_k\bigg\} \, .
\end{align}
This condition will mean that edges between vertices along the cut will be informative as we try to synchronize them.

Next we recursively define the set of good level $k$ blocks $\cG^{(k)}$.  A block $B\in \cB^{(k)}$ is \emph{good} if 
\begin{itemize}
  \item There is at most one bad  $(k-1)$-level sub-block of $B$, that is
  \begin{align}
  \left|\left\{B_i \in \cB^{(k-1)}: B_i \subset B, B_i \not\in \cG^{(k-1)} \right\}\right| \leq 1\, .
  \end{align}
  \item All level $k-1$ sub-block edges are honest, 
  \begin{align}
  \bigcap_{\substack{B_1,B_2 \in \cB^{(k-1)}\\ B_1,B_2 \subset B,\; B_1\sim B_2}} \cA^{(k-1)}(B_1,B_2)\, .
  \end{align}
\end{itemize}

\begin{claim}
There exists $p_\star > 0$ such that, if $0<p<p_\star$ then for all $B \in \cB^{(k)}$
\begin{equation}\label{e:goodBound}
\prob(B\in \cG^{(k)}) \geq 1 - 2^{-200k-200}.
\end{equation}
\end{claim}
\begin{proof}
We will establish~\eqref{e:goodBound} inductively.  Note that blocks at level 0 are good.  First we estimate the probability that the honest edge condition holds.  Assuming that $p_\star \leq \frac1{40}$,
\begin{align*}
\prob\left(\cA^{(k-1)}(B_1,B_2)\right) &= \prob\left( \hbox{Bin}(\ell_{k-1},1-p) \geq \frac{9}{10} \ell_{k-1}\right)\\
&\ge \prob\left( \hbox{Bin}(\ell_{k-1},\frac{39}{40}) \geq \frac{9}{10} \ell_{k-1}\right) \ge 1- \exp\left(-\kappa 2^{10k(k-1)}\right)
\end{align*}
for some $\kappa >0$.  Thus
\begin{equation}\label{e:edgeHonestBound}
\prob\left(\cA^{(k-1)}(B_1,B_2)\right) \ge 1- 2^{-400k-800}
\end{equation}
for all sufficiently large $k$.  By taking $p_\star$ small enough equation~\eqref{e:edgeHonestBound} holds for small $k$ as well and thus for all $k$.  Hence, since there are $2^{40k}$ level $k-1$ sub-blocks in each level $k$ block we have that,
\begin{equation}\label{e:edgeHonestUnionBound}
\prob\left(\bigcap_{\substack{B_1,B_2 \in \cB^{(k-1)}\\ B_1,B_2 \subset B}} \cA^{(k)}(B_1,B_2)\right) \geq 1-2^{40k+1} \cdot 2^{-400(k-1)-800} \geq 1-2^{-200k-201} \, .
\end{equation}
Since there are no bad sub-blocks at level 0 this implies~\eqref{e:goodBound} for $k=1$.
For some $k\geq 2$, assume inductively that equation~\eqref{e:goodBound} holds up to $k-1$.  Then, since the event that blocks are good are independent, for $B \in \cB^{(k)}$,
\begin{align*}
\prob\left(\left|\left\{B' \in \cB^{(k-1)}: B' \subset B, B' \notin \cG^{(k-1)} \right\}\right| \geq 2\right) 
&= \prob\left( \hbox{Bin}(2^{40k},2^{-200(k-1)-200}) \geq 2 \right)\\
&\leq {2^{40k} \choose 2} (2^{-200k})^2 \leq 2^{-320k} \leq 2^{-200k-240}\, .
\end{align*}
Combining with equation~\eqref{e:edgeHonestUnionBound} we have that
\[
\prob(B\in \cG^{(k)}) \geq 1 - 2^{-200k-200} \, ,
\]
as required.
\end{proof}

Next we describe how to inductively construct the synchronization variables $W_{B}^{(k)}$ in a $k+1$ block $B^*$.  For $B_1\sim B_2$ $k$-level sub-blocks of $B^*$ we let 
\[
\bY_{B_1,B_2} = \mathrm{sign}\left(\sum_{\substack{B_1\ni x \sim y\in B_2}} \tilde{W}_{x}^{(k)} \tilde{W}_{y}^{(k)} \bY_{xy}\right)
\]
We assign the $W_{B}^{(k)}$ as follows:
\begin{enumerate}
  \item A quartet is a collection of 4 sub-blocks $B_1\sim B_2\sim B_3\sim B_4\sim B_1$ that form a square of side length $2\ell_k$.  A quartet is incoherent if $\prod_{i=1}^{4} \bY_{B_i,B_{i+1}}=-1$ where we take $B_5 = B_1$.  Let $\cI^{(k)}_{B^*}$ be the set of sub-blocks of $B^*$ that appear in no incoherent quartets.  It is possible for $\cI^{(k)}_{B^*}$ to be disconnected, in that case take $\cI^{(k)}_{B^*}$ to be the largest component.
  \item If possible, assign $W_{B}^{(k)}$ for all $B\in \cI^{(k)}_{B^*}$  such that for all adjacent sub-blocks $B_1, B_2\in \cI^{(k)}_{B^*}$ we have that 
\begin{equation}\label{e:synchCondition}
        W_{B_1}^{(k)}W_{B_2}^{(k)}=\bY_{B_1,B_2}  
      \end{equation}
      Denote the event that such an assignment is possible as $\cH^{(k+1)}_{B^*}$.  If such an assignment is not possible set all the $W_{B}^{(k)}=1$.  Set $W_{B}^{(k)}=1$ for 
all $B\in (\cI^{(k)}_{B^*})^c$.
\end{enumerate}
In the following we will write $\cI = \cI^{(k)}= \cI^{(k)}_{B^*}$ omitting arguments when clear from the context.
Note that on the event $\cH^{(k+1)}_{B^*}$, the $W_{B}^{(k)}$ can be found efficiently by assigning the variables iteratively to satisfy equation~\eqref{e:synchCondition}.

\begin{claim}\label{claim2}
For $k\geq1$, if $B \in \cG^{(k)}$ is good then the following hold:
\begin{enumerate}
\item $\cH^{(k)}_{B}$ holds.  
\item There exists a random variable $S^{(k)}_B \in\{-1,1\}$ such that if $x\in B$ and on the event
\begin{align}
\bigcap_{j=0}^{k-1} \Big\{\{D_{x,j} \in\cG^{(j)}\}\cap \{D_{x,j} \in \cI^{(j)}\}\Big\}
\end{align}
we have that
\begin{equation}\label{e:recoveryS}
\btheta_x = S^{(k)}_B \tilde{W}_{x}^{(k)} \, .
\end{equation}
\item Furthermore, for any $B'\in \cG^{(k)}$ with $B' \sim B$,
\begin{equation}\label{e:goodEdge}
\sum_{\substack{x \in B\cap \partial B'}} S^{(k)}_B \tilde{W}_{x}^{(k)}\btheta_x \geq (1-2^{-8} +2^{-10k} )\ell_k\, .
\end{equation}
(Here $\partial B' \equiv \{x\in\integers^2:\, {\rm dist}(x,B')=1\}$.)
\end{enumerate}
\end{claim}
Note that we do not (and cannot) construct $S^{(k)}_B$ and observe that it is used in the analysis but not the construction.  It accounts for the fact that we can only hope to recover the $\btheta_u$ up to a global multiplicative shift.
\begin{proof}[Proof of Claim \ref{claim2}]
We proceed inductively.  In the base case when $k=0$ for $x=B\in \cG^{(0)}$ we may set $S^{(0)}_x = \btheta_x$. With the convention that an empty product is 1 we have that $\tilde{W}_{x}^{(0)}=1$ and so
\[
\btheta_x = S^{(0)}_x \tilde{W}_{x}^{(0)} \, .
\]
Now we assume the claim holds for all $k' < k$ and consider  a good block $B\in \cG^{(k)}$. 

\noindent{\bf 1.} For any good $(k-1)$-level sub-blocks, $B_1\sim B_2$ in $B$
\begin{align}\label{e:edgeSynch}
\bY_{B_1,B_2} &= \mathrm{sign}\left(\sum_{\substack{B_1\ni x \sim y\in B_2}} \tilde{W}_{x}^{(k-1)} \tilde{W}_{y}^{(k-1)}\bY_{xy} \right)\nonumber\\
&=  \mathrm{sign}\left(S^{(k-1)}_{B_1} S^{(k-1)}_{B_2}\sum_{\substack{B_1\ni x \sim y\in B_2}} S^{(k-1)}_{B_1} \tilde{W}_{x}^{(k-1)} S^{(k-1)}_{B_2}\tilde{W}_{y}^{(k-1)} \bY_{xy} \right)
\end{align}
Our inductive hypothesis implies that there are at most $2^{-8}\ell_{k-1}$ vertices $x$ in this sum with $S^{(k-1)}_B \tilde{W}_{x}^{(k-1)}\neq \btheta_x$, thus
\begin{align}
\sum_{\substack{B_1\ni x \sim y\in B_2}} S^{(k-1)}_{B_1} \tilde{W}_{x}^{(k-1)} S^{(k-1)}_{B_2}\tilde{W}_{y}^{(k-1)} \bY_{xy} \geq \sum_{\substack{B_1\ni x \sim y\in B_2}} \btheta_x \btheta_y \bY_{xy} - 4\cdot 2^{-8}\ell_{k-1}\, ,
\end{align}
and so since $\cA^{(k-1)}(B_1,B_2)$ holds,
\begin{align}
\sum_{\substack{B_1\ni x \sim y\in B_2}} S^{(k-1)}_{B_1} \tilde{W}_{x}^{(k-1)} S^{(k-1)}_{B_2}\tilde{W}_{y}^{(k-1)} \bY_{xy} \geq \left(\frac{9}{10} - 4\cdot 2^{-8}\right)\ell_{k-1} > 0\, .
\end{align}
Combining with equation~\eqref{e:edgeSynch} we have that
\begin{align}
\bY_{B_1,B_2}  = \mathrm{sign}\left(S^{(k-1)}_{B_1} S^{(k-1)}_{B_2}\right).
\end{align}
It follows that every quartet of good sub-blocks is coherent.  If all of the $(k-1)$-level quartets of sub-blocks of $B$ are coherent then there are exactly two assignments of $W_{B_i}^{(k-1)}$ (related by a multiplicative factor of $-1$) satisfying $W_{B_1}^{(k-1)}W_{B_2}^{(k-1)}=\bY_{B_1,B_2}$.  If there is one or more incoherent quartet, this must include the single bad sub-block.  The sub-blocks in $\cI$ are good and there exist two assignments  satisfying $W_{B_1}^{(k)}W_{B_2}^{(k)}=\bY_{B_1,B_2}$ for all $B_1,B_2 \in \cI$ which are,
\begin{equation}\label{epairOf Solutions}
W_{B_i}^{(k-1)} \equiv S^{(k-1)}_{B_i}\quad \hbox{or} \quad W_{B_i}^{(k-1)} \equiv -S^{(k-1)}_{B_i}.
\end{equation}
In either case the procedure will construct $W_{B_i}^{(k)}$ satisfying~\eqref{epairOf Solutions} on $\cI$ and $\cH^{(k)}_{B}$ holds.  We set $S^{(k)}_{B}$ so that
\[
S^{(k)}_{B} W_{B_i}^{(k-1)} \equiv S^{(k-1)}_{B_i}.
\]

\noindent{\bf 2.} To verify condition~\eqref{e:recoveryS} we see that for $x\in B_i$,
\[
S^{(k)}_B \tilde{W}_{x}^{(k)} = S^{(k)}_B W_{B_i}^{(k)} \tilde{W}_{x}^{(k-1)} = S^{(k-1)}_{B_i}\tilde{W}_{x}^{(k-1)} =\btheta_x \, ,
\]
where the last equality used the inductive hypothesis. 

\noindent{\bf 3.} It remains to check the condition on the boundary of $B$ adjacent to some good block $B'$. Since any sub-block in $\cI^c$ must be in a quartet with a bad sub-block, there are at most 3 on any side of $B$.  Thus, summing over sub-blocks $B_i$ of $B$ we have that
\begin{align*}
\sum_{\substack{x\in B \cap \partial B'}} S^{(k)}_B \tilde{W}_{x}^{(k)}\btheta_x 
&= \sum_{B_i:B_i\sim B'}\sum_{\substack{x\in B_i \cap \partial B'}} S^{(k)}_B \tilde{W}_{x}^{(k)}\btheta_x \\
&\geq \sum_{B_i\in \cI:B_i\sim B'}\sum_{\substack{B_i\ni x \sim y\in B'}} S^{(k)}_B \tilde{W}_{x}^{(k)}\btheta_x -3\ell_{k-1}\\
&\geq (1-2^{-8} +2^{-10(k-1)} )\ell_{k-1}(2^{20k}-3)-3\ell_{k-1}\\
&\geq (1-2^{-8} +2^{-10k})\ell_{k}
\end{align*}
which establishes~\eqref{e:goodEdge}.
\end{proof}
By the proceeding claim, if $u$ and $v$ are in the same $k$-level block on the event
\[
\cJ_{uv}^{(k)} = \bigcap_{j=0}^{k-1} \Big\{\{D_{u,j},D_{v,j} \in\cG^{(k)}\}\cap \{D_{u,j},D_{v,j} \in \cI\}\Big\}
\]
we have that
\begin{equation}\label{e:recovery}
\tilde{W}_{u}^{(k-1)}\tilde{W}_{v}^{(k-1)} = \btheta_u  S^{(k)}_B \btheta_v  S^{(k)}_B =\btheta_u\btheta_v \, .
\end{equation}
so $\tilde{W}_{u}^{(k-1)}\tilde{W}_{v}^{(k-1)}$ correctly recovers $\btheta_u\btheta_v$.  A sufficient condition for $D_{u,k}\in \cG^{(k)}\cap \cI$ is that $D_{u,k}$ and the 8 sub-blocks surrounding it are all good.  Thus
\[
\prob(\cJ_{uv}^{(k)}) \geq 1 - \sum_{k'\geq 1} 18\,\prob(D_{u,k}\in \cG^{(k)}) \geq 1 - 18\sum_{k'\geq 1} 2^{-200k-200}\geq \frac{9}{10}.
\]
Thus
\[
\prob(\tilde{W}_{u}^{(k-1)}\tilde{W}_{v}^{(k-1)} = \btheta_u\btheta_v) \geq \frac{8}{10}
\]
and so the success probability of recovery is at least $\frac{8}{10}> \frac12$ independent of the distance between $u$ and $v$ which completes the proof of Theorem~\ref{t:Z2NoiseZ2}.

\section*{Acknowledgements}

This work was partially supported by the following grants:  NSF CAREERAward CCF-1552131, NSF CSOI CCF-0939370 (E.A.); NSF CCF-1319979, NSF
DMS-1613091 (A.M); NSF CCF-1553751 and a Sloan Research Fellowship (N.S).
We thank the American Institute of Mathematics (San Jose, CA) where part of this work was carried out.

\appendix

\section{Proof of Eq.~(\ref{eq:Unbiased}) for $\O(m)$ synchronization}
\label{app:Lambda}

Here we prove the remark that --under the model of  Example \ref{ex:Orth}-- $\E\{\bY_{xy}|\btheta\} = \lambda(\sigma^2)\, \btheta_x^{-1}\btheta_y$.
Fixing for simplicity $x=1$, $y=2$ and dropping the indices $x,y$ unless necessary, we have $\bY = \tbU\tbV^{\sT}$
where $\tbX= \btheta_1^{-1}\btheta_2+\sigma\bZ$ has singular value decomposition $\tbX =\tbU\bSigma\tbV^{\sT}$. 

Let $\bX = \btheta_1\tbX\btheta_2^{-1} = \bU\bSigma\bV^{\sT}$. Our claim is equivalent to $\E\{\bU\bV^{\sT}|\btheta\} =\lambda(\sigma^2)\id$.  
By rotational  invariance  of the Gaussian distribution, we have $\bX = \id+\sigma \, \tilde{\bG}$ for $(G_{ij})_{1\le i,j\le m}\sim_{iid}\normal(0,1)$
or --equivalently-- $\bX = \bQ^{\sT}(\id+\sigma \, \bG)\bQ$ for any $\bQ$ in $\O(m)$. Using the last representation,
$\bE = \E\{\bU\bV^{\sT}|\btheta\} = \E\{\bQ^{\sT}\bU\bV^{\sT}\bQ|\btheta\}$ for $\id+\sigma \, \bG =\bU\bSigma\bV^{\sT}$. This implies that
$\bQ^{\sT}\bE\bQ=\bE$ for any orthogonal matrix $\bQ$, which can hold only if $\bE =\lambda\id$ for some scalar $\lambda$.

Continuity and the limit values of $\lambda(\sigma^2)$ are straightforward.

\section{Proof of Theorem \ref{thm:Noisy}}

Let $p \equiv 1- \inf_{\by,\btheta_0}  q(\by|\btheta_0)$. 
We can write the conditional probability density $q(\by|\btheta_0)$ as
\begin{align}
q(\by|\btheta_0) = (1-p)+p\, q_*(\by|\btheta)\, .
\end{align}
Hence observations $(\bY_{xy})_{(x,y)\in E}$ can be generated as follows. First draw independent random
variables $(U_{xy})_{(x,y)\in E}\sim_{iid}{\rm Bernoulli}(p)$. Then, for each $(x,y) \in E$ such that $U_{xy}=1$, draw an independent 
observation $\bY_{xy}\sim q_*(\,\cdot\, |\btheta_x^{-1}\btheta_y)$. For $(x,y) \in E$ such that $U_{xy}=0$,
draw $\bY_{xy}$ according to the Haar measure.

To upper bound the total variation distance in Eq~(\ref{eq:WeakRecovery}) we consider the easier problem 
in which instead of $\bY$, we are given all the Bernoulli variables $U = (U_{xy})_{(x,y)\in E}$ and, for each $(x,y)\in E$
such that $U_{xy}=1$ we are given the group difference $\bD_{xy}=\btheta_x^{-1}\btheta_y$. Denoting by $\bD=\{\bD_{xy}\}_{(x,y)\in E, U_{xy}=1}$,
we then have
\begin{align}
\Big\|\prob\big(\btheta_xT_{xy}(\bY)\btheta_y^{-1}\in \;\cdot\;\big)- \prob_{{\rm Haar}}\big( \; \cdot\;\big)\Big\|_{\sTV} 
\le \sup_{\tT_{xy}}\Big\|\prob\big(\btheta_xT_{xy}(U;\bD)\btheta_y^{-1}\in \;\cdot\;\big)- \prob_{{\rm Haar}}\big( \; \cdot\;\big)\Big\|_{\sTV} \, .
\end{align}
Consider the percolation process defined by the variables $U$ (whereby
edge $(x,y)\in E$ is open if $U_{xy}=1$), and denote by  $x\conn y$ the event that $x$ and $y$ are
in the same percolation cluster. If $x$ and $y$ are not in the same percolation cluster, then the conditional distribution
of $\btheta_x^{-1}\btheta_y$ conditional on $U;\bD$ is uniformly on $\Group$.  This implies that
\begin{align}
\Big\|\prob\big(\btheta_xT_{xy}(\bY)\btheta_y^{-1}\in \;\cdot\;\big)- \prob_{{\rm Haar}}\big( \; \cdot\;\big)\Big\|_{\sTV} \le \prob(x\conn y)\, .
\end{align}
For $p<p_c(d)$, the right hand side goes to $0$ as $\|x-y\|\to \infty$ \cite{grimmett1989percolation}, which yields the claim.

\section{Proof of Theorem \ref{t:Z2ctsNonRecon}}

For $s\in \reals$ and $Z\sim g(\,\cdot\,)$,  we define
\begin{align}
\psi(s) = \|\prob\big(Z e^{i s} \in \,\cdot\,\big)- \prob\big(Z  \in \,\cdot\,\big) \|^2_{L^2(g)} &= \int_0^{2\pi}\left(\frac{g(e^{i (t-s)})}{g(e^{ i (t)})} - 1 \right)^2 g(e^{ i (t)}) \; \de t  \nonumber\\
&=  \int_0^{2\pi}\left(\frac{g(e^{i (t-s)})}{g(e^{i (t)})}\right)^2 g(e^{i (t)}) \;\de t -1 \, .
\end{align}
Note that $\psi(s)$ is twice differentiable, nonnegative and that $\psi(0)=0$ so $\psi'(s)=0$ and for some $\kappa=\kappa(g)>0$,
\begin{align}
|\psi(s)| \leq \kappa |s|^2.
\end{align}

Let $u,v\in \integers^2$ with $L=\|u-v\|_2$, and define the function $h:\integers^2\to \reals$ by
\begin{align}
h(x) = 1-\frac{\log\big(1+\min(\|x-u\|_2;L)\big)}{\log(L+1)}\, .
\end{align}
Note that $h(u)=1$, $h(v) = 0$ and $h(x)=0$ for $\|x-u\|_2\ge L$. Fix $\btheta \in \U1^{\integers^2}$, $s\in [0,2\pi)$ and define $\btheta^{(s)}$
by letting $\btheta_x^{(s)} = e^{is}\, \btheta_x$. Denote by $\prob_{\btheta}(\bY\in \,\cdot\, )$ the conditional distribution of the observations given hidden 
variables $\btheta$.
We then have, for a constant $C$,
\begin{align}
\Big\|\prob_{\btheta^{(s)}}\big(\bY \in \,\cdot\,\big)- \prob_{\btheta}\big(\bY \in \,\cdot\,\big)\Big\|^2_{\sTV}
&\leq \Big\|\prob_{\btheta^{(s)}}\big(\bY \in \,\cdot\,\big)- \prob_{\btheta}\big(\bY \in \,\cdot\,\big)\Big\|^2_{L^2}\\
&=\prod_{(x,y)\in E} \bigg( 1+ \Big\|\prob_{\nu_s}\big(\bY_{xy} \in \,\cdot\,\big)- \prob_{\nu_0}\big(\bY_{xy} \in \,\cdot\,\big)\Big\|^2_{L^2}\bigg) -1\\
&=\prod_{(x,y)\in E} \big( 1+ \psi\big(s(h(x)-h(y))\big)\big) -1\\
&\leq \prod_{(x,y)\in E} \big( 1+ \kappa\, s^2 \, |h(x)-h(y)|^2)\big) -1\\
&\leq \prod_{\substack{(x, y)\in E \\ \|x-u\| \leq L}} \left\{ 1+ \frac{C}{(1+\|x-u\|^2)\log^2 L}\right\} -1\\
&=O(1/\log L).
\end{align}
Taking expectation over $s$ uniformly random in $[0,2\pi)$ (denoted by $\rE_s$), we have, for any measurable set $B$, 
\begin{align}
\rE_s\prob_{\btheta^{(s)}}\big(\btheta_uT_{uv}(\bY)\btheta_v \in B\big) =\rE_s\prob_{\btheta^{(s)}}\big(\btheta^{(s)}_uT_{uv}(\bY)\btheta^{(s),-1}_v \in e^{is}\, B\big)\, ,
\end{align}
and therefore
\begin{align}
\Big|\rE_s\prob_{\btheta^{(s)}}\big(\btheta^{(s)}_uT_{uv}(\bY)\btheta^{(s),-1}_v \in e^{is}\, B\big)-\prob_{\btheta}\big(\btheta_uT_{uv}(\bY)\btheta_v^{-1} \in B\big)\Big| = O(\log(1/L))\, .
\end{align}
We next take expectation with respect to $(\btheta_x)_{x\in\integers^2}$ i.i.d. uniform in $\U1$. Note that under this distribution, also 
$(\btheta_x)_{x\in\integers^2}$ are i.i.d. uniform in $\U1$. Letting $\prob(\,\cdot\,) = \E\prob_{\btheta}(\,\cdot\,)$, we have
\begin{align}
\Big|\rE_s\prob\big(\btheta_uT_{uv}(\bY)\btheta^{-1}_v \in e^{is}\, B\big)-\prob\big(\btheta_uT_{uv}(\bY)\btheta_v^{-1} \in B\big)\Big| = O(\log(1/L))\, .
\end{align}
For any fixed $B$, $\bxi$, $\rP_s(\bxi\in e^{is}B) = \prob_{Haar}(B)$ and hence we get 
\begin{align}
\Big|\prob_{\btheta}\big(\btheta_uT_{uv}(\bY)\btheta_v^{-1} \in B\big)-\prob_{{\rm Haar}}(B)\Big| = O(\log(1/L))\, .
\end{align}
This proves the impossibility of weak recovery.

\bibliographystyle{amsalpha}

\newcommand{\etalchar}[1]{$^{#1}$}
\providecommand{\bysame}{\leavevmode\hbox to3em{\hrulefill}\thinspace}
\providecommand{\MR}{\relax\ifhmode\unskip\space\fi MR }
\providecommand{\MRhref}[2]{%
  \href{http://www.ams.org/mathscinet-getitem?mr=#1}{#2}
}
\providecommand{\href}[2]{#2}

\end{document}